\documentclass[11pt]{amsart}%
\usepackage{palatino, mathpazo}
\usepackage{amsfonts}
\usepackage{amsmath}
\usepackage{amssymb,latexsym,xcolor}
\usepackage{graphicx}
\usepackage{harpoon}
\usepackage[mathscr]{eucal}
\usepackage{amssymb}%
\usepackage[all]{xy}
\usepackage[
linktocpage=true,colorlinks,citecolor=magenta,linkcolor=blue,urlcolor=magenta]{hyperref}
\providecommand{\U}[1]{\protect \rule{.1in}{.1in}}

\newtheorem{theorem}{Theorem}[section]
\newtheorem{theorem*}{Theorem}

\newtheorem{conjecture}[theorem]{Conjecture}
\newtheorem{corollary}[theorem]{Corollary}
\newtheorem{definition}[theorem]{Definition}

\newtheorem{lemma}[theorem]{Lemma}
\newtheorem{proposition}[theorem]{Proposition}
\newtheorem{Theorem}{Theorem}

\theoremstyle{remark}
\newtheorem{remark}[theorem]{Remark}

\newtheorem{question}{Question}

\numberwithin{equation}{section}

\newcommand{\R}{{\mathbb{R}}}

\newcommand{\Z}{{\mathbb{Z}}}
\newcommand{\C}{{\mathbb{C}}}

\newcommand{\n}{{\mathbf{n}}}
\renewcommand{\l}{{\lambda}}

\begin{document}
\title[Classification of spherical metrics on tori]{Classification of spherical metrics on tori with four singularities, I: half periods }
\author{Erjuan Fu}
\address{Beijing Institute of Mathematical Sciences and Applications, Beijing, 101408, China}
\email{fej.2010@tsinghua.org.cn, ejfu@bimsa.cn}
\author{Chang-Shou Lin}
\address{Department of Mathematics, National Taiwan University, Taipei 10617, Taiwan }
\email{cslin@math.ntu.edu.tw}

\begin{abstract}
Classifying the spherical metrics on a torus $E_\tau$  with $4\pi$ conic angle at each half period point\, ${\omega_k}/{2}, k=0,1,2,3$\,  is equivalent to classify solutions of the following curvature equation
\begin{align}\label{eq0731093154}
    \Delta u+e^u=4\pi\sum_{k=0}^3\delta_{\frac{\omega_k}{2}}\text{\ on\ }E_\tau
\end{align}
where $\tau\in \mathbb{H}:=\{z\in \mathbb{C}\mid \mathrm{Im} \, z>0\}$ and $\delta_p$ is the Dirac measure at $p\in E_\tau$. 
 By constructing a multiple Green function 
$$G_2(z_1, z_2;\tau):=G(z_1-z_2;\tau)-\frac{1}{2}\sum_{j=0}^3\left(G(z_1-\frac{\omega_j}{2};\tau)+G(z_2-\frac{\omega_j}{2};\tau)\right),  $$
in terms of the Green function $G(z;\tau)$ on $E_\tau$, 
we classify the solutions of (\ref{eq0731093154}) into two types:  \emph{special} and  \emph{non-special}. Furthermore, we obtain the following conclusion about the solutions of (\ref{eq0731093154}):
\begin{enumerate}
\item any special solution is an even function and 
the set of special solutions is isomorphic to $SL(2,\mathbb{C})/SU(2)$ for all $\tau\in \mathbb{H}$.
\item   a non-special solution exists if and only if $\tau\in \mathcal{E}$.
Moreover, if $\tau\in \mathcal{E}$, then
there are  six one-parameter families of nonspecial solutions. 
\end{enumerate}
where
$$\mathcal{E}:=\left\{\tau\in \mathbb{H}\mid G(z;\tau) \,\,\text{has exactly 5 critical points.}\right\}.$$
The set $\mathcal{E}$ is completely determined in \cite{CLW2018, Lin}, which  is a union of countable many open triangular domains.
As a byproduct, we  completely determine and classify the critical points of 
$G_2$ and then obtain the degeneracy criterion of critical points for $G_2$, 
which  may be of independent interest.
\end{abstract}

\thanks{Erjuan Fu was supported by NSFC (No. 12401188) and BIMSA Start-up Research Fund.}

\maketitle
\tableofcontents
\setcounter{tocdepth}{1}

\section{Introduction}\label{sec-intro}


We start with a classical problem in geometry: finding the spherical metric on punctured Riemann surface with conic singularities. 
By a \emph{spherical metric}, we mean a metric of constant positive curvature, in particular, we let the curvature identically equal $1$. 

To be specific, 
let $S$ be a compact Riemann surface of genus $g$
and $$\Sigma:=\{p_1, \cdots, p_n\}\subseteq S$$ be a finite set,  
we focus on how to understand (decribe, classify)  the \emph{spherical metric} $\rho(z)|dz|$ on $S\setminus \Sigma$ with \emph{conic singularities} along $\Sigma$, i.e., $\rho(z)$ is positive smooth on $S\setminus \Sigma$ and satifies 
\begin{align*}
&\Delta \log \rho+ \rho^2=0 \quad \text{on} \, \, S\setminus \Sigma,\\
&\rho(z)\sim |z-p_j|^{\alpha_j-1}, \quad (z\to p_j),
\end{align*}
where $ \alpha_j, \, j=1,\cdots, n$ are positive numbers, i.e., $2\pi \alpha_j$ is the total {angle} around the singularity $p_j$.
By letting $u=\log(2\rho^2)$, the problem is equivalent to study 
the following curvature equation 
\begin{equation}\label{eqn-mfe}
\Delta u+ e^{u}=4\pi \sum\limits_{j=1}^n(\alpha_j-1)\delta_{p_j} \quad \text{on} \, \, S,
\end{equation}
where $\delta_{p_j}$ is the Dirac measure at $p_j\in S$. 

%

Since (\ref{eqn-mfe}) possesses the so-called bubbling phenomena (see \cite{ChenL2003,ChenL2015,ChenL2022}), the existence or non-existence problems are considered difficult questions to be solved, unless the parameter $\alpha:=(\alpha_1, \cdots, \alpha_n)$ satisfies some non-critical conditions. In this case, a degree counting formulas for solutions of (\ref{eqn-mfe}) had been derived (see \cite{ChenL2015,ChenL2022}) and solutions of (\ref{eqn-mfe}) exist if the degree does not vanish. For the critical case, up to now, there are only a few results to gurantee the existence of solutions of (\ref{eqn-mfe}).
Recently the problem 
has been getting a lot of attention from mathematician with different research backgrounds. Indeed, the problem can be investigated from the aspects of differential geometry \cite{EG2015,MP2016,MP2019}, the complex function theory \cite{Eremenko2004, Eremenko2020, EG2015},  monodromy theory and premodular theory \cite{CLW2015,CL2018,CL2020,CL2021,FL2025,Kuo,Lin,LW2010,
LW2017,LW2017-1,LW2020,LY2018}, algebraic geometry \cite{EGMP2022,LSX2021,Tarantello2010},  and partial differential equations \cite{ChenL2015,ChenL2022,MP2019}. 

We consider the case of $g=1$, i.e., $S=E_\tau:=\mathbb{C}/\Lambda_\tau$ is a torus, where $\Lambda_\tau=\mathbb{Z}+\mathbb{Z}\tau$ and   $\tau\in \mathbb{H}:=\{z\in \mathbb{C} \mid \mathrm{Im}z>0\}$. 
 For the case of $g=0$, we refer to the recent works \cite{Eremenko2004,Eremenko2020,EG2015,EGT2014}. 
%
%
%
%
%
%
%
The curvature equation (\ref{eqn-mfe}) becomes 
\begin{equation}\label{mfe-torus}
\Delta u+e^u=4\pi \sum_{j=1}^n(\alpha_j-1) \delta_{p_j} \qquad \text{on}\,\, E_\tau.
\end{equation}
In a series of papers (this paper is the first one), we are only insterested in the situtation when $0<\alpha_j-1\in \mathbb{N}^+$ for all $j$. If $\sum_{j=1}^n(\alpha_j-1)$ is an odd integer, then equation (\ref{mfe-torus}) belongs to the category of non-critical case. Thus, as mentioned earlier, the degree counting formulas, derived in \cite{ChenL2015} and \cite{ChenL2022}, can be applied to study (\ref{mfe-torus}). Therefore, we always assume that 
\begin{equation}\label{even-cond}
\sum_{j=1}^n(\alpha_j-1) \in 2\mathbb{N}^+
\end{equation}

There is an interesting observation concerning (\ref{mfe-torus})  under the assumption (\ref{even-cond}). To investigate the equation (\ref{mfe-torus}), we find the concept of developing map  is particularly useful. By the uniformization theorem, any solution $u(z)$ can be obtained through a developing map $h(z)$:
\begin{equation}\label{uz}
u(z)=\log \frac{8|h'(z)|^2}{(1+|h(z)|^2)^2}
\end{equation}
In the literature, formula (\ref{uz}) is called the Liouville theorem. Under our assumption on $\alpha$, it is not hard to see that the developing map $h(z)$ can always be chosen so that it is meromorphic on $\mathbb{C}$ and satisfy 
\begin{equation}\label{hz}
h(z+\omega_j)=\lambda_jh(z), j=1,2, \quad \lambda_j\in \mathbb{C} \, \text{and} \, |\lambda_j|=1,
\end{equation}
where $\omega_1:=1$ and $\omega_2:=\tau$. 
Obviously, if (\ref{hz}) holds for $h(z)$, then it also holds for $e^th(z)$ with $t\in \mathbb{R}$. By substituting $e^th(z)$ to (\ref{uz}), we obtain a real one-parameter family of solutions: \footnote{In the rest of this paper, we will omit ``real" and always use real numbers to parametrize a family of solutions.
} 
\begin{equation}\label{ut}
u_t(z)=\log \frac{8e^{2t}|h'(z)|^2}{(1+e^{2t}|h(z)|^2)^2}.
\end{equation}
%
%
 A consequece is the a priori bound of (\ref{mfe-torus}) does not exist. This is one of the reasons why it is difficult to solve (\ref{mfe-torus}) by applying analytic techniques only. 
We call $u_t$ is generated by $u$. 
 Generally, 
given a solution  $v$ of (\ref{mfe-torus}) with developing map $g$, we call $v$  is  generated by $u$  if 
$$g=Rh:=\frac{ah+b}{ch+d}\qquad  \text{for some}\,\, R:=\begin{pmatrix}a&b\\ c&d\end{pmatrix}\in SL(2, \mathbb{C}).$$  
This observation was noticed in \cite{CLW2015} and \cite{LW2010}, where $n=1$ and $\alpha_1=3$ were assumed. Indeed, if $n=1$ and $\alpha_1:=2m+1\in 2\mathbb{N}^++1$, then the second author proved the following result.

\begin{Theorem}[\cite{Lin}]
Let $m\in \mathbb{N}^+$ and consider the curvature equation 
\begin{equation}\label{mfe-8}
\Delta u+e^u=8\pi m\delta_0 \qquad \text{on}\quad E_\tau.
\end{equation}
Then \begin{enumerate}
\item any solution $u$ of (\ref{mfe-8}) generates a one-parameter family of solutions;
\item the set $\{\tau\in \mathbb{H}\mid (\ref{mfe-8}) \,\text{has a solution.}\}$ is an open subset of $\mathbb{H}$. 
\end{enumerate}
\end{Theorem}

In this paper, we focus on the following curvature equation:
\begin{equation}\label{mfe}
\Delta u+e^u=4\pi \sum_{j=0}^3\delta_{\omega_j/2}\qquad \text{on}\quad E_\tau,
\end{equation}
where $\omega_0:=0$ and $\omega_3:=1+\tau$. 
Surprisingly, there is always a solution  $u(z)$ of (\ref{mfe}) such that it generates a three-parameter family of solutions. 
Indeed, for any $\tau\in \mathbb{H}$, we could find a solution $u_0$ with  developing map $h_0(z)$ such that the factors $\lambda_1,\lambda_2$ in (\ref{hz}) are equal to 1,  then $Mh_0$ satisfies (\ref{hz}) for any $M\in SL(2,\mathbb{C})$.
 By the Liouville theorem, we get a solution $u_M$ of (\ref{mfe}) with $Mh_0$ as a developing map and then $u_M$  is generated by $u_0$.
Furthermore, $u_{M_1}=u_{M_2}$ if and only if $M_1=UM_2$ for some $U\in SU(2)$.  Then the family $\{u_M\}$ generated by $u$ is a real three-parameter family, because  $SL(2,\mathbb{C})/SU(2)$ is a three dimensional real analytic manifold.

\begin{definition}\label{def-u}
Throughout the paper, a solution of  (\ref{mfe}) is called \emph{special} solution if it  generates  a  three-parameter family of solutions; 
 and a solution of (\ref{mfe}) is called \emph{nonspecial} if  it  generates  a   one-parameter family of solutions. 
\end{definition}
One of our main theorems is to classify all the solutions of (\ref{mfe}). 

\begin{theorem}\label{thm-1.1}
Let $\tau\in\mathbb{H}$. 
We have the following conclusion about the solutions of (\ref{mfe}).
\begin{enumerate}
\item Any solution is either special or nonspecial. 
\item Any special solution is an even function and 
the set of special solutions is isomorphic to $SL(2,\mathbb{C})/SU(2)$.
\item The curvature equation (\ref{mfe})  has a nonspecial solution if and only if the following curvature equation 
\begin{equation}\label{mfe-0}
\Delta u+e^u=8\pi \delta_{0}\qquad \text{on}\quad E_\tau,
\end{equation} has a solution. 
Moreover, if (\ref{mfe-0}) possesses a solution, then (\ref{mfe})  exactly contains six one-parameter families of nonspecial solutions. 
\end{enumerate}
\end{theorem}

Theorem \ref{thm-1.1} (1) is proved in Theorem \ref{lem-gen-u}, Theorem \ref{thm-1.1} (2) is proved in Theorem \ref{cor-special}  and Theorem \ref{thm-1.1} (3) is proved in  Theorem \ref{cor-nontrivial}. 
 Notice that  the right hand side of (\ref{mfe}) is symmetric with respect to $\frac{\omega_j}{4}, \, 0\leq j\leq 3$. By Theorem \ref{thm-1.1} (3), we might ask whether each one-parameter family of non-special solutions contains a unique solution, which is symmetric to $\omega_j/4$.  Theorem \ref{thm-sys} completely answer this question. The existence of symmetric solutions is useful when we discuss the bubbling phenomenon of (\ref{mfe-0}) during deforming $\tau\in\mathbb{H}$.

In \cite{LW2010}, they proved that (\ref{mfe-0}) has no solutions if    $\tau=bi$ with $b>0$, i.e., 
$E_\tau$ is rectangular. Together with Theorem \ref{thm-1.1} (2), the nonexistence result yields the following result.

\begin{corollary}
If $\tau=bi$ with $b>0$, then (\ref{mfe}) only has special solutions.
\end{corollary}

We solve Theorem \ref{thm-1.1} by noticing that  the curvature equation (\ref{mfe}) is closely related to critical points of some multiple Green function. This idea is inspired by  \cite{CLW2015,CFL2025-1}.
Let $G(z)$ be the Green function on $E_\tau$ with signularity at $0$.   In \cite{CFL2025-1},  in order to consider the following curvature equation 
\begin{equation}\label{mfe-p}
\Delta u+e^u=4\pi(\delta_{p}+\delta_{-p})\qquad \text{on}\quad E_\tau,
\end{equation} 
where $p\in E_\tau\setminus E_\tau[2]$ and  $E_\tau[2]:=\{\omega_j/2\in E_\tau \mid j=0,1,2,3\}$,
a generalized Green function $G_p(z)$ 
with singularities at $\pm p$ is constructed by 
$$G_p(z)=\frac{1}{2}(G(z+p)+G(z-p)),\qquad z\in E_\tau.$$
For any $\tau \in \mathbb{H}$, $G_p(z)$ always has its critical point at each half period $\omega_j/2$, which are called trivial critical points, and all other critical points are called nontrivial. It is clear that the nontivial critical points appear in pair, i.e., $q$ is a nontrivial critical point of $G_p$ if and only if $-q$ is a nontrivial critical point of $G_p$. 

\begin{Theorem}\cite{CFL2025-1}\label{thm-CFL2025-1}
Let $\tau \in \mathbb{H}$ and $k:=k(p;\tau)$ be the number of pairs of nontrivial critical points of $G_p(z;\tau)$, then $k\in \{0,1,2,3\}$ and all numbers in $\{0,1,2,3\}$ can be taken when $\tau,\,p$ vary. 
Furthermore, all solutions of (\ref{mfe-p}) are nonspecial and it contains exactly $k$ one-parameter family of nonspecial solutions. 
\end{Theorem}

Theorem \ref{thm-CFL2025-1} is not the first one in the literature to find the connection between curvature equations and Green functions. In \cite{LW2010}, there proves the Green function $G(z)$ has either 3 critical points (all are trivial critical points) or 5 critical points (two extra nontrivial critical points), and moreover, the curvature equation (\ref{mfe-0}) has a unqiue one-parameter family of nonspecial solutions if and only if $G(z)$ has a pair of nontrivial critical points.  
Thus, Theorem \ref{thm-CFL2025-1} is a natural generalization of the result to (\ref{mfe-p}). 

In order to classify solutions of (\ref{mfe}),  we construct the following  associated multiple Green function: 
\begin{equation}\label{def-G2}
G_2(z_1, z_2):=G(z_1-z_2)-\frac{1}{2}\sum_{j=0}^3\left(G(z_1-\frac{\omega_j}{2})+G(z_2-\frac{\omega_j}{2})\right),  
\end{equation}
where $z_1, z_2\in E_\tau\setminus E_\tau[2]$ and $z_1\neq z_2$. It is clear that  $G_2(z_1, z_2)=G_2(z_2, z_1)$ because $G$ is an even function.  Hence, $G_2$ can be viewed as a function of  the set $\{z_1, z_2\}\in \mathrm{Sym}^2(E_\tau\setminus E_\tau[2])\setminus \Delta_2$, where $\Delta_2=\{\{z,z\}\,\mid\,z\in E_\tau\}$. 

 Similar to $G(z)$ and $G_p(z)$, 
$G_2(z)$ always has the following 12 critical points for any $\tau\in \mathbb{H}$: 
$$[(\frac{\omega_t}{4}, \frac{\omega_t}{4}-\frac{\omega_\ell}{2})], [(-\frac{\omega_t}{4}, -\frac{\omega_t}{4}-\frac{\omega_\ell}{2})], \, t\in \{j, k\}\, \text{and}\, \{j, k, \ell \}=\{1,2,3\},$$
which are called \emph{trivial critical points}, and 
the other critical points are called \emph{nontrivial critical points}. However, the structure of nontrivial critical points of $G_2(z)$ is more complicated than  $G(z)$ and $G_p(z)$.  For the sake of convenience, we classify the nontrivial critical points of $G_2(z)$ into two classes. By noticing that $\{z, -z\}$ with   $z\in E_\tau\setminus E_\tau[2]$ is always a nontrivial critical point of $G_2$,  this type of nontrivial critical points $\{z,-z\}$ are called  \emph{special critical points} and the other nontrivial  critical points are called \emph{nonspecial critical points}.

In Section \ref{sec-G2}, we will construct an one-parameter of solutions from a nonspecial critical point of $G_2$. By this specific construction, two equivalent non-special critical points would construct a same one-parameter family of solutions.
With this construction,
we make a correspondence between the nontrivial critical points of $G_2$ and the solutions of (\ref{mfe}). In this correspondence, special critical points correspond to special solutions and nonspecial critical points correspond to nonspecial solutions. Moreover, the correspondence is one to one between the equivalence class of nonspecial critical points and the set of one-parameter family of nonspecial solutions. We therefore finish the proof of Theorem \ref{thm-1.1}.


Finally, we consider the degeneracy criterion of critical points for $G_2$. Recently, it was proved in \cite{CFL2025-1} that the pair nontrivial critical points $\pm \theta$ of $G$ are nondegenerate with $\det D^2G(\pm \theta)>0$. 
Denote by
\begin{equation}\label{def-E}
\mathcal{E}:=\left\{\tau\in \mathbb{H}\mid G(z;\tau) \,\,\text{has nontrivial critical points.}\right\},
\end{equation}
which is a union of countable many open triangular domains \cite{CLW2018,Lin}.  Our second main result is the following.
\begin{theorem}\label{thm-intro-deg}
Let $\tau\in \mathbb{H}$,  then all special critical points  of $G_2$ are degenerate and  the followings hold: 
\begin{enumerate}
\item 
If $\tau\in \mathcal{E}$, then 
\begin{itemize}
\item[(1a)]  all nonspecial critical points $\{z_1, z_2\}$ of $G_2$ are nondegenerate and 
 $\det D^2G_2(z_1,z_2)<0$; 
\item[(1b)] all trivial critical points $\{z_1, z_2\}$ of $G_2$ are nondegenerate and\\
$\det D^2G_2(z_1,z_2)>0$.
\end{itemize}
\item If $\tau\in \mathbb{H}\setminus\mathcal{E}$, then there are no nonspecial critical points.  Futhermore, 
\begin{enumerate}
\item [(2a)] if $\tau\in \mathbb{H}\setminus \overline{\mathcal{E}}$, then all trivial critical points of $G_2$  are nondegenerate.
\item [(2b)] if $\tau\in \partial{\mathcal{E}}$, then exactly $4$ trivial critical points  of $G_2$ are degenerate and $\det D^2G_2(z_1,z_2)>0$ for the remaining $8$ trivial critical points $\{z_1,z_2\}$. 
\end{enumerate}
\end{enumerate}
\end{theorem}


This information of nondegeneracy is important. 
There is an application of Theorem \ref{thm-intro-deg}.
\begin{Theorem}\label{thm-app}
For any $\tau \in \mathcal{E}$, there is a small number $\varepsilon(\tau)>0$ such that the curvature equation 
$$\Delta u+e^u=4\pi \sum_{k=0}^3 \delta_{p_k}\qquad \text{on}\,\, E_\tau$$
has at least six one-parameter families of solutions, where $p_0=0$ and $|p_j-\frac{\omega_j}{2}|<\varepsilon(\tau), 1\leq j\leq 3$. 
\end{Theorem}
This theorem and some other applications will be discussed in a coming paper.

The paper is organized as follows. We first recall the important tool  -- developing maps Section \ref{sec-developing} and some results about the generalized Lam\'e equations in Section \ref{sec-GLE}.     Then we discuss the critical points of the multiple Green function $G_2$  in Section \ref{sec-G2} and finith the proof of Theorem \ref{thm-1.1} in Section \ref{sec-degeneracy}. Finally, we prove Theorem \ref{thm-intro-deg} in Section \ref{sec-deg}.

\section{Developing maps and associated ODEs}\label{sec-developing}

In order to study the spherical metrics with conic signularities,
we first introduce a very important tool -- the developing map.

Let $u$ be a solution of (\ref{eqn-mfe}). Notice that 
a surface $S\setminus \Sigma$ of curvature $1$ is locally isometric to a region on the standard sphere which means the Riemann sphere ${\mathbb{P}}^1:=\mathbb{C}\cup \{\infty\}$ with Fubini-Study metric $2|dw|/(1+|w|^2)$.  This local isometry can be analytically continued along any path not passing through the singularities so we have a multivalued locally biholomorphic map 
$$f: S\setminus \Sigma \to \mathbb{P}^1.$$
That is, $f$ can be lifted to a holomorphic map $\widetilde{f}$ defined on a (not necessarily connected) covering space $\pi: \widetilde{S\setminus \Sigma}\to S\setminus \Sigma$ and the following diagram commutes
$$
\xymatrix{\widetilde{S\setminus \Sigma}\ar[d]_\pi\ar[r]^{\widetilde{f}}&\mathbb{P}^1\\
S\setminus \Sigma\ar[ur]_{f}}.
$$
Since $f$ is a local isometry,  then the metric $\rho(z)|dz|$ on $S\setminus \Sigma$ is recovered from $f$ by the formula
\begin{equation}\label{rho-f}
\rho(z)=\frac{2|f'(z)|}{1+|f(z)|^2},
\end{equation}
thus 
\begin{equation}\label{u-f}
u(z)=\log\frac{8|f'(z)|^2}{(1+|f(z)|^2)^2}.
\end{equation}
Notice that  $\rho(z)\sim |z-p_j|^{\alpha_j-1}, \, (z\to p_j)$, then
\begin{equation}\label{faj}
f(z)\sim (z-p_j)^{\alpha_j},\quad (z\to p_j).
\end{equation}
A map $f$ satisfying (\ref{u-f}) is called a \emph{developing map} of $u$ (or $\rho$). 

However, the developing map is not unique. Indeed, since the group of automorphisms preserving the Fubini-Study metric on $\mathbb{P}^1$ is given by $SU(2)$,  
then
two developing maps $f$ and $g$ define the same spherical metric $\rho$ if and only if 
$g=Rf$ where $$R=\begin{pmatrix}a&b\\ c&d\end{pmatrix}\in SU(2) \quad \text{and} \quad Rf=\frac{af+b}{cf+d}.$$
Therefore, 
 classification of the solutions of (\ref{eqn-mfe})
 is equivalent to  classification of developing maps modulo the action of $SU(2)$. 

Consider the \emph{Schwarzian derivative} $S[f]$ of $f$, which is defined by
$$S[f]:=\frac{f'''}{f'}-\frac{3}{2}\left(\frac{f''}{f'}\right)^2.$$
From (\ref{u-f}), we have 
\begin{equation}
S[f]=u_{zz}-\frac{1}{2}u_z^2,
\end{equation}
which means that $S[f]$ does not depend on the choice of $f$.
Since $u$ satisfies (\ref{eqn-mfe}), then
$$S[f]_{\overline{z}}=u_{zz\overline{z}}-u_zu_{z\overline{z}}=0,$$
thus $S[f]$ is holomphic on $S\setminus \Sigma$. 
 The behavior (\ref{faj})  of $f$ near the conic singularities $p_j$ implies that $S[f]$ has double poles at these points and 
\begin{equation}\label{localSf}
S[f]\sim \frac{1-\alpha_j^2}{2}(z-p_j)^{-2}, \quad (z\to p_j).
\end{equation}
Hence, $S[f]$ is 
a rational function in the case of the sphere, and an elliptic function in the case of torus.

Consider the second order differential equation of Fuchsian type induced by $S[f](z;\tau)=u_{zz}-\frac{1}{2}u_z^2$:
\begin{equation}\label{lu}
\mathcal{L}(u): 
\quad y''(z)=-\frac{1}{2}S[f](z;\tau)y(z),\quad z\in \mathbb{C}.
\end{equation}
Let $y_1(z), y_2(z)$ be two linearly independent solutions of $\mathcal{L}(u)$, then $S\left[\frac{y_1}{y_2}\right]=S[f]$.
Notice that $S[g]=S[f]$ if and only if $g=Rf$ for some $R\in SL(2, \mathbb{C})$, then there exists 
$R\in SL(2,\mathbb{C})$ such that   $$f=R\left(\frac{y_1}{y_2}\right)=\frac{ay_1+by_2}{cy_1+dy_2},\quad \text{where} \,\,R=\left(\begin{array}{cc}a&b\\ c&d\end{array}\right).$$ 
Denote by $f_1=ay_1+by_2$ and $f_2=cy_1+dy_2$ which are another pair of linearly independent solutions of $\mathcal{L}(u)$, we have $f=f_1/f_2$. 
Therefore,  classification of developing maps modulo the action of $SU(2)$ is equivalent to  classification of linearly independent solutions of the corresponding Fuchsian equations modulo the action of  $SU(2)$.  Consequently, 
 classification of the solutions $u$ of (\ref{eqn-mfe})
 is equivalent to  classification of linearly independent solutions of $\mathcal{L}(u)$ modulo the action of  $SU(2)$.

In the rest of this paper, we focus on  the case of $g=1$, i.e., $S=E_\tau:=\mathbb{C}/\Lambda_\tau$ is a torus, where $\Lambda_\tau=\mathbb{Z}+\mathbb{Z}\tau$ and   $\tau\in \mathbb{H}:=\{z\in \mathbb{C} \mid \mathrm{Im}z>0\}$. 

Let $u$ be a solution of (\ref{mfe-torus}) and 
$f$ be a developing map of $u$, then $S[f]$ is an elliptic function and satisfies (\ref{localSf}), then we can set 
\begin{equation}\label{eqn-Sf}
-\frac{1}{2}S[f](z;\tau)=\sum_{j=1}^n\frac{\alpha_j^2-1}{4}\wp(z-p_j;\tau)+\sum_{j=1}^nT_j\zeta(z-p_j;\tau)+B,
\end{equation}
with $T_j,\, B\in \mathbb{C}$ and $\sum_{j=1}^nT_j=0$. 
Here, $\wp(\cdot; \tau), \zeta(\cdot;\tau)$ are the Weierstrass elliptic and zeta functions with basic periods 1 and $\tau$, respectively.  
%
%
Since $T_j$ and $B$ depend on the solution $u$, they are not completely free. Indeed,  $(T_1, \cdots, T_n, B)$ is called apparent parameters, which means all solutions of $\mathcal{L}(u)$ are free of logarithmic singularity at all singularity $p_1, \cdots, p_n$, and we call $\mathcal{L}(u)$ is apparent correspondingly. See the following lemma for explanation.

\begin{lemma}\label{lem-app}
Let $u$ be a solution of (\ref{mfe-torus}), then $\mathcal{L}(u)$ is apparent.
\end{lemma}

\begin{proof}
 Let $y_1(z), y_2(z)$ be two linearly independent solutions  of  $\mathcal{L}(u)$ such that the developing map $f(z)$ of $u$ satisfy 
$f(z)=y_1(z)/y_2(z)$. If $\mathcal{L}(u)$ is not apparent, then at least one of $y_1(z), y_2(z)$ has a logarithmic at some singulartiy $p_k$.  Notice that the Wronskian $W:=y_1y_2'-y_1'y_2$ of $y_1, y_2$ is a nonzero constant,
 then
$$e^u=\frac{8|f'(z)|^2}{(1+|f(z)|^2)^2}=\frac{8|W^2|}{|y_1(z)|^2+|y_2(z)|^2}$$
is not single-valued near $p_k$, which is a contradiction, because $u$ is single valued on $E_\tau\setminus \Sigma$.
\end{proof}

\section{Generalized Lam\'e equations with apparent parameters}\label{sec-GLE}
From now on, we focus on $\Sigma=E_\tau[2]$ and $\alpha=(2,2,2,2)$,
and consider the following  generalized Lam\'e equation 
\begin{equation}\label{ODE-0}
y''(z)=q(z;\mathbb{T}, B, \tau)y(z),\quad z\in \mathbb{C},
\end{equation}
where 
\begin{align*}
q(z; \mathbb{T},B,\tau)=\sum_{j=0}^3\frac{3}{4}\wp(z-\frac{\omega_j}{2};\tau)+\sum_{j=0}^3T_j\zeta(z-\frac{\omega_j}{2};\tau)+B,
\end{align*}
and $(\mathbb{T}, B):=(T_1,\cdots, T_4, B)$ is apparent. 
Recall some results about (\ref{ODE-0}) from \cite{FL2025}, we have 
(\ref{ODE-0}) is apparent at $\omega_k/2$ 
if and only if $(\mathbb{T},B)$ satisfies 
\begin{equation}\label{cond-AP}
T_k^2-\frac{3}{4}\sum_{j\neq k}\zeta_{kj}T_j-B=0,
\end{equation}
where  $\zeta_{kj}=\zeta(\frac{\omega_k}{2}-\frac{\omega_j}{2})$, and  the space of all apparent parameters is
$$AP=\left\{(\mathbb{T}, B)\,\mid \, \sum_{j=0}^3T_j=0 \,\,\text{and}\,\, (\ref{cond-AP}) \,\,\text{holds for}\,\, k=0,1,2,3\right\}.$$
%
In particular, $AP$ can be decomposed into three irreducible components. 
\begin{theorem}\cite[Lemma 2.1]{FL2025}
The apparent space  $AP=\cup_{j=1}^3 V_j$, where 
$$V_j=\left\{(\varepsilon_0^{(j)}T, \varepsilon_1^{(j)}T,\varepsilon_2^{(j)}T,\varepsilon_3^{(j)}T,B)\,\mid\, B=T^2-\eta_{3-j}T\,\,\text{and}\,\, T\in \mathbb{C}\right\}\cong \mathbb{C}$$
with $\eta_0=0$, $\eta_j:=\zeta(z+\omega_j)-\zeta(z), j=1,2,3$ are the quasi-periods of $\zeta(z)$ and $\varepsilon_0^{(j)}=\varepsilon_j^{(j)}=-\varepsilon_k^{(j)}=-\varepsilon_\ell^{(j)}=1$ for $\{j,k,\ell\}=\{1,2,3\}$.
\end{theorem}
%
Notice that any component $V_j$ of $AP$ can be identified with $\mathbb{C}$, then we can simply use one variable $T$ to denote an element $(\mathbb{T}, B)\in AP$. 
To simplify notation, we denote (\ref{ODE-0}) by $\mathcal{L}(T)$ if $T\in AP$ and $\mathcal{L}_\ell (T)$ if $T\in V_\ell$  for some $\ell\in \{1,2,3\}$. 
Specifically, let $T\in V_\ell$, then (\ref{ODE-0}) becomes 
\begin{equation*}
\mathcal{L}_\ell(T):\qquad y''(z)=q_\ell(z;T)y(z),\quad z\in \mathbb{C},
\end{equation*}
where 
\begin{align*}
q_\ell (z;T)=\sum_{j=0}^3\frac{3}{4}\wp(z-\frac{\omega_j}{2};\tau)+T\sum_{j=0}^3\varepsilon_j^{(\ell)}\zeta(z-\frac{\omega_j}{2};\tau)+T^2-\eta_{3-\ell}T.
\end{align*}

\begin{lemma}\label{lem-even}
Let $\ell \in \{1,2,3\}$ and $\{\ell ,j,k\}=\{1,2,3\}$. Then
\begin{enumerate}
\item $q_\ell(z;T)$ has a period $\frac{\omega_\ell}{2}$, i.e., $q_\ell(z+\frac{\omega_\ell}{2};T)=q_\ell(z;T)$. 
\item 
$q_\ell(z;T)$ is even 
    with respect to $\frac{\omega_j}{4}$ and $\frac{\omega_k}{4}$, i.e., $q_\ell(\frac{\omega_j}{4}-z)=  q_\ell(\frac{\omega_j}{4}+z)$ and  $q_\ell(\frac{\omega_k}{4}-z)=  q_\ell(\frac{\omega_k}{4}+z)$.
\end{enumerate}
\end{lemma}
\begin{proof}
First, by considering $\ell=j+k, \, k=\ell+j,\, j=\ell+k$ respectively, we obtain that 
\begin{align*}
    &q_\ell (\frac{\omega_\ell}{2}+z)- q_\ell (z)\\
=&T\left(\sum\limits_{i=0}^{3}\epsilon_i^{(\ell)}\zeta(z+\frac{\omega_\ell}{2}-\frac{\omega_i}{2})-\sum\limits_{i=0}^{3}\epsilon_i^{(\ell)}\zeta(z-\frac{\omega_i}{2})\right)\\
=&T\left(\zeta(z+\frac{\omega_\ell}{2})-\zeta(z-\frac{\omega_\ell}{2})+\zeta(z-\frac{\omega_j}{2})-\zeta(z+\frac{\omega_\ell}{2}-\frac{\omega_j}{2})\right.\\
&\left.+\zeta(z-\frac{\omega_k}{2})-\zeta(z+\frac{\omega_\ell}{2}-\frac{\omega_k}{2})\right)\equiv 0.
\end{align*}
Hence, (1) is proved. In order to prove (2), we consider 
\begin{align*}
    q_\ell (\frac{\omega_j}{4}+z)- q_\ell (\frac{\omega_j}{4}-z)=&\frac{3}{4}\sum\limits_{i=0}^{3}\left(\wp(z+\frac{\omega_j}{4}+\frac{\omega_i}{2})-\wp(z-\frac{\omega_j}{4}+\frac{\omega_i}{2})\right)\\&+T\sum\limits_{i=0}^{3}\epsilon_i^{(\ell )}\left(\zeta(z+\frac{\omega_j}{4}-\frac{\omega_i}{2})+\zeta(z-\frac{\omega_j}{4}+\frac{\omega_i}{2})\right).
\end{align*}
Since 
\begin{align*}
&\sum\limits_{i=0}^{3}\left(\wp(z+\frac{\omega_j}{4}-\frac{\omega_i}{2})-\wp(z-\frac{\omega_j}{4}+\frac{\omega_i}{2})\right)\\
=&\wp(z+\frac{\omega_j}{4})+\wp(z+\frac{\omega_j}{4}-\frac{\omega_j}{2})+\wp(z+\frac{\omega_j}{4}-\frac{\omega_\ell}{2})+\wp(z+\frac{\omega_j}{4}-\frac{\omega_k}{2})\\
-&\wp(z-\frac{\omega_j}{4})-\wp(z-\frac{\omega_j}{4}+\frac{\omega_j}{2})-\wp(z-\frac{\omega_j}{4}+\frac{\omega_\ell}{2})-\wp(z-\frac{\omega_j}{4}+\frac{\omega_k}{2})\equiv 0
\end{align*}
and 
\begin{align*}
   &\sum\limits_{i=0}^{3}\epsilon_i^{(\ell)}\left(\zeta(z+\frac{\omega_j}{4}-\frac{\omega_i}{2})+\zeta(z-\frac{\omega_j}{4}+\frac{\omega_i}{2})\right)\\
=&\zeta(z+\frac{\omega_j}{4})+\zeta(z+\frac{\omega_j}{4}-\frac{\omega_\ell}{2})-\zeta(z+\frac{\omega_j}{4}-\frac{\omega_j}{2})-\zeta(z+\frac{\omega_j}{4}-\frac{\omega_k}{2})\\
+&\zeta(z-\frac{\omega_j}{4})+\zeta(z-\frac{\omega_j}{4}+\frac{\omega_\ell}{2})-\zeta(z-\frac{\omega_j}{4}+\frac{\omega_j}{2})-\zeta(z-\frac{\omega_j}{4}+\frac{\omega_k}{2})\\
=& \zeta(z+\frac{\omega_j}{4}-\frac{\omega_\ell}{2})-\zeta(z+\frac{\omega_j}{4}-\frac{\omega_k}{2})+\zeta(z-\frac{\omega_j}{4}+\frac{\omega_\ell}{2})-\zeta(z-\frac{\omega_j}{4}+\frac{\omega_k}{2}),
\end{align*}
which is identically equal to $0$ by considering $\ell=j+k, \, k=\ell+j,\, j=\ell+k$ respectively, then $ q_\ell (\frac{\omega_j}{4}+z)=q_\ell (\frac{\omega_j}{4}-z)$. The same proof is applied to $q_\ell(\frac{\omega_k}{4}-z)= q_\ell(\frac{\omega_k}{4}+z)$. Therefore, (2) is proved.
    \end{proof}

Notice that the local exponents of $\mathcal{L}_\ell(T)$ at the singularity $\omega_k/2$ are $-1/2$ or $3/2$, then the local monodromy matrix around  $\omega_k/2$ is $-I_{2}$, so we only need to consider the monodromy matrices of $\mathcal{L}_\ell(T)$  along the two fundamental cycles. 
Fix a point $z_0\in \mathbb{C}$ such that the potential $q_{\ell}(z;T)$ has no singularities along the lines $z_0+\mathbb{R}$ and $z_0+\tau\mathbb{R}$. Let $y_1(z)$ and $y_2(z)$ be two linearly independent solutions of $\mathcal{L}_{\ell}(T)$ near $z_0$. Since $q_{\ell}(z;T)$ is elliptic, there are two  $2\times2$ matrices $M_j(z_0),j=1,2$, such that
\begin{align*}
    \begin{pmatrix}y_1(z+\omega_j)\\y_2(z+\omega_j) \end{pmatrix}=M_j(z_0)\begin{pmatrix}
    y_1(z)\\y_2(z) \end{pmatrix} 
\end{align*}
for $z$ in a neighborhood of $z_0+\mathbb{R}$ or $z_0+\tau\mathbb{R}$. 
Notice that,  
for different $z_0'$, we have $M_j(z_0')=\epsilon_jM_j(z_0)$ for some $\epsilon_j\in\{\pm1\}$ because $\alpha_k-1\in \mathbb{N}, k=1, \cdots, n$. 
Hence, without loss of generality, we can fix the base point $z_0$ and then simply denote the monodromy matrices by $M_1,M_2$.  Furthermore,   
we always have $
M_1M_2=M_2M_1$.
Therefore, the monodromy representation can be split into two cases: \emph{completely reducible case and not completely reducible case.}
The equation $\mathcal{L}_{\ell}(T)$ is called completely reducible if $M_j,j=1,2$, can be simultaneously diagonalized. Otherwise, it is called not completely reducible.
%
%
If $\mathcal{L}_{\ell}(T)$ is completely reducible, up to a common conjugation, we can set \begin{equation}\label{CompRed}M_1=\begin{pmatrix} e^{-2\pi is}&0\\ 0& e^{2\pi is}\end{pmatrix},\quad M_2=\begin{pmatrix} e^{2\pi ir}&0\\ 0& e^{-2\pi ir}\end{pmatrix},\quad s,r\in \mathbb{C}.\end{equation}  Here, $M_1, M_2$ can be in $\{\pm I\}$ simultaneously.
If $\mathcal{L}_{\ell}(T)$  is not completely reducible, up to a common conjugation, we can set 
\begin{equation}\label{CompNRed}
M_1=\varepsilon_1\begin{pmatrix} 1&0\\ 1& 1\end{pmatrix},\quad M_2=\varepsilon_2\begin{pmatrix} 1&0\\ \mathcal{C}& 1\end{pmatrix},\end{equation} 
with $\varepsilon_1, \varepsilon_2\in \{\pm 1\}$ and $\mathcal{C} \in \mathbb{C}\cup \{\infty\}$.  If $\mathcal{C}=\infty$, then (\ref{CompNRed}) is understood as 
\begin{equation*}
M_1=\varepsilon_1\begin{pmatrix} 1&0\\ 0& 1\end{pmatrix},\quad M_2=\varepsilon_2\begin{pmatrix} 1&0\\ 1& 1\end{pmatrix}.\end{equation*} 
 Throughout the paper, $(s,r) \mod \mathbb{Z}^2$ and $\mathcal{C}$ are called the \emph{monodromy data} of  $\mathcal{L}_{\ell}(T)$   for each case.  
 
By introducing the following so-called spectral polynomial  $Q_\ell(T)$  for $\mathcal{L}_\ell(T)$:
   \begin{equation*}\label{cc12}
  Q_\ell(T)=16(T^2+e_\ell-e_j)(T^2+e_\ell-e_k),
    \end{equation*}
where $\{\ell, j, k\}=\{1,2,3\}$ and $e_i=\wp(w_i/2), i=1,2,3$,
we obtain that $\mathcal{L}_\ell(T)$ is completely reducible if and only if $Q_\ell(T)\neq 0$. See  \cite{FL2025}  for details. 

Let $u$ be a solution of (\ref{mfe}), by Lemma \ref{lem-app}, there exists $T\in V_\ell$ for some $\ell\in \{1,2,3\}$ such that $\mathcal{L}(u)=\mathcal{L}_\ell(T)$, which is in fact
 unitarizable. Here ``unitarizable" means that the monodromy matrices can be unitary by a common conjugation.  
The converse also holds, see the following lemma for details. 

\begin{lemma}\label{lem-unit}
Let $u$ be a solution of (\ref{mfe}), then $\mathcal{L}(u)$ is unitarizable. Conversely, let $T\in V_\ell$ for some $\ell\in \{1,2,3\}$,   if  $\mathcal{L}_\ell(T)$ is unitarizable, then 
there exists a solution $u$  of (\ref{mfe}) such that $\mathcal{L}(u)=\mathcal{L}_\ell(T)$.
\end{lemma}

\begin{proof}
On the one hand, let $u$ be a solution of (\ref{mfe})
and $f$ be a developing map of $u$, then there is a pair of linearly independent solutions $f_1(z), f_2(z)$ of $\mathcal{L}(u)$ such that $f=f_1/f_2$.  
Notice that $f_1(z), f_2(z)$ are multiple-valued with respect to $z$ and might have branch points at $E_\tau[2]$.
For any loop $\gamma\in \pi_1(E_\tau\setminus E_\tau[2])$, denote by $\gamma^*f_k(z)$ the analytic continuation of $f_k(z)$ along the loop $\gamma$, then there exists a matrix $M_\gamma^f \in \mathrm{SL}(2, \mathbb{Z})$ such that 
$$\gamma^*\left(\begin{array}{c}f_1(z)\\f_2(z)\end{array}\right)=\left(\begin{array}{c}\gamma^*f_1(z)\\ \gamma^*f_2(z)\end{array}\right)=M_\gamma^f \left(\begin{array}{c}f_1(z)\\f_2(z)\end{array}\right).$$ 
%
%
By (\ref{u-f}), we have 
\begin{equation}\label{eqn-norm-f}
\left(\begin{array}{cc}\overline{f_1(z)}&\overline{f_2(z)}\end{array}\right)\left(\begin{array}{c}f_1(z)\\f_2(z)\end{array}\right)=|f_1|^2+|f_2|^2=2\sqrt{2} e^{-\frac{1}{2}u}|W|,
\end{equation}
\begin{equation}\label{eqn-norm-f1}
\left(\begin{array}{cc}\overline{f_1(z)}&\overline{f_2(z)}\end{array}\right)\overline{M_\gamma^f}^TM_\gamma^f \left(\begin{array}{c}f_1(z)\\f_2(z)\end{array}\right)=2\sqrt{2} e^{-\frac{1}{2}\gamma^*u}|W|,
\end{equation}
where $W$ is the Wronskian of $f_1, f_2$. Notice that $W$ is a nonzero constant and  $u$ is single valued on $E_\tau\setminus E_\tau[2]$, i.e.,  $\gamma^*u=u$.  By (\ref{eqn-norm-f}), (\ref{eqn-norm-f1}) and taking the partial derivative with respect to $z, \overline{z}$, we get 
$\overline{M_\gamma^f}^TM_\gamma^f=I_{2}$, i.e., $M_\gamma^f\in SU(2)$, so $\mathcal{L}(u)$ is unitarizable. 

On the other hand, let $T\in V_\ell$ for some $\ell\in \{1,2,3\}$ and  $\mathcal{L}_\ell(T)$ is unitarizable,  
then $\mathcal{L}_\ell(T)$  is completely reducible and we can pick 
 two linearly independent solutions $y_1(z)$, $y_2(z)$ such that   the monodromy matrices $M_1, M_2$ under this basis satisfy (\ref{CompRed})
with $s, r\in \mathbb{R}$. 
  Let $f(z)=\frac{y_1(z)}{y_2(z)}$, we have 
\begin{equation}\label{eqn-f-1}f(z+\omega_k)=\lambda_k^2f(z), \quad k=1,2,
\end{equation}
where $\lambda_1=e^{-2\pi is}$ and $\lambda_2=e^{2\pi ir}$.
Then 
$$u(z)=\log \frac{8|f'(z)|^2}{(1+|f(z)|^2)^2}$$
is well-defined on $E_\tau\setminus E_\tau[2]$ and solve 
$$\Delta u+e^u=0,\quad \text{on}\,\, E_\tau\setminus E_\tau[2].$$

Since $T\in V_\ell$ and  the local exponents of $\mathcal{L}_\ell(T)$ at $\frac{\omega_k}{2}$ are $-\frac{1}{2},\frac{3}{2}$, then the two fundamental solutions locally can be expressed as one of the following forms:
\begin{align*}
&\textbf{Case 1.}\quad y_1(z)=(z-\frac{\omega_k}{2})^{\frac{3}{2}}g_1(z),\qquad y_2(z)=(z-\frac{\omega_k}{2})^{-\frac{1}{2}}g_2(z),\\
&\textbf{Case 2.}\quad y_1(z)=(z-\frac{\omega_k}{2})^{-\frac{1}{2}}g_1(z),\qquad y_2(z)=(z-\frac{\omega_k}{2})^{\frac{3}{2}}g_2(z),\\
&\textbf{Case 3.}\quad y_1(z)=(z-\frac{\omega_k}{2})^{-\frac{1}{2}}g_1(z),\qquad y_2(z)=(z-\frac{\omega_k}{2})^{-\frac{1}{2}}g_2(z),
\end{align*}
where $g_1(z), g_2(z)$ are holomorphic near $\frac{\omega_k}{2}$ and $g_j(\frac{\omega_k}{2})\neq 0, j=1,2$. 
All the three cases gives us 
$$f(z)= \frac{y_1(z)}{y_2(z)}
\sim  \text{const. }+(z-\frac{\omega_k}{2})^{\pm 2},\quad (z\to \frac{\omega_k}{2}),$$
 and then 
$$u(z)\sim -2\log |z-\frac{\omega_k}{2}|,\quad (z\to \frac{\omega_k}{2}),$$ which means $u$ is a solution of (\ref{mfe}) and then $\mathcal{L}(u)=\mathcal{L}_\ell(T)$. 
\end{proof}

Notice that $\mathcal{L}_\ell(T)$ is unitarizable if and only if it  is  completely reducible and $s, r\in \mathbb{R}$.  By the definition of special and nonspecial solutions (see Definition \ref{def-u}), we can classify solutions of  (\ref{mfe}) into special solutions and nonspecial solutions in terms of the monodromy data $(s,r)$. 

\begin{theorem}\label{lem-gen-u}
Let $u$ be a solution of (\ref{mfe}) and $(s,r)$ be the monodromy data of $\mathcal{L}(u)$, then 
 $u$ is special if $(s,r)\in (\frac{1}{2}\mathbb{Z})^2$ and nonspecial if $(s, r)\in \mathbb{R}^2\setminus (\frac{1}{2}\mathbb{Z})^2$. Moreover, 
\begin{enumerate}
\item if $(s,r)\in  (\frac{1}{2}\mathbb{Z})^2$, then 
(\ref{mfe}) has exactly a three-parameter family of solutions
$v$ such that $\mathcal{L}(v)=\mathcal{L}(u)$, i.e., $u$ generates  a three-parameter family of solutions. 
\item if $(s,r)\in  \mathbb{R}^2\setminus (\frac{1}{2}\mathbb{Z})^2$, 
 then (\ref{mfe}) has exactly a one-parameter family of solutions $v$ such that $\mathcal{L}(v)=\mathcal{L}(u)$, i.e., $u$ generates  an  one-parameter family of solutions. 
\end{enumerate}
\end{theorem}

\begin{proof}
Let $u$ be a solution of (\ref{mfe}) and  
 $T\in V_\ell$ for some $\ell\in \{1,2,3\}$ such that $\mathcal{L}(u)=\mathcal{L}_\ell(T)$ is unitarizable,
then $\mathcal{L}_\ell(T)$  is completely reducible and we can pick 
 two linearly independent solutions $y_1(z)$, $y_2(z)$ such that   the monodromy matrices $M_1, M_2$ under this basis satisfy (\ref{CompRed})
with $s, r\in \mathbb{R}$. 
  Let $f(z)=\frac{y_1(z)}{y_2(z)}$, we have 
\begin{equation}\label{eqn-f-1}f(z+\omega_k)=\lambda_k^2f(z), \quad k=1,2,
\end{equation}
where $\lambda_1=e^{-2\pi is}$ and $\lambda_2=e^{2\pi ir}$.
Then 
$$v(z)=\log \frac{8|f'(z)|^2}{(1+|f(z)|^2)^2}$$
is a solution of (\ref{mfe}) and    $\mathcal{L}(v)=\mathcal{L}_\ell(T)$ by the proof of Lemma \ref{lem-unit}.

If $(s,r)\in  (\frac{1}{2}\mathbb{Z})^2$, then the monodromy matrices under any basis are  $\pm I_{2}$.
 By a similar analysis for $Rf$ instead of $f$ in the proof of Lemma \ref{lem-unit},  we have 
$$v_R(z):=\log \frac{8|(Rf)'(z)|^2}{(1+|Rf(z)|^2)^2}$$
is also a solution of (\ref{mfe}) and $\mathcal{L}(v_R)=\mathcal{L}_\ell(T)$. 
Let $w$ be a solution of (\ref{mfe}) satisfying $\mathcal{L}(w)=\mathcal{L}_\ell(T)$ and $g$ be a developing map of $w$, then
$S[g]=S[f]$, thus there exists some $R\in SL(2, \mathbb{C})$ such that $g=Rf$, i.e.,  $w(z)=v_R(z)$. 
By a direct computation, $v_{R_1}(z)=v_{R_2}(z)$  if and only if $R_2=PR_1$ for some $P\in SU(2)$. 
 Therefore, the set of solutions of (\ref{mfe}) such that the monodromy data $(s,r)\in  (\frac{1}{2}\mathbb{Z})^2$  is isomorphic to $SL(2, \mathbb{C})/SU(2)$.

If $(s,r)\in \mathbb{R}^2\setminus (\frac{1}{2}\mathbb{Z})^2$, 
without loss of generality, we can assume that $s\notin  \frac{1}{2}\mathbb{Z}$.
Notice that  $f(z)=\frac{y_1(z)}{y_2(z)}$, we can replace the eigenfunction $y_1(z)$ with a nonzero constant multiple $\lambda y_1(z)$, where $\lambda \in \mathbb{C}\setminus \{0\}$. Denote by $e^t:=|\lambda|$ with $t\in \mathbb{R}$, we have 
$$v_{t}(z)=\log \frac{8e^{2t}|f'(z)|^2}{(1+e^{2t}|f(z)|^2)^2}$$
is a solution of (\ref{mfe}) and    $\mathcal{L}(v_t)=\mathcal{L}_\ell(T)$. By a direct computation, $v_{t_1}(z)=v_{t_2}(z)$ if and only if $t_1=t_2$.  Let $w$ be a solution of (\ref{mfe}) satisfying $\mathcal{L}(w)=\mathcal{L}_\ell(T)$ and  $g$ be a developing map of $w$, then   $S[g]=S[f]$, thus 
 $$g=\frac{ay_1+by_2}{cy_1+dy_2}=\frac{af+b}{cf+d},\quad \text{for some}\,\, R:=\left(\begin{array}{cc}a&b\\ c&d\end{array}\right)\in SL(2, \mathbb{C}),$$
which gives us 
$$M_k^g(\tau)=RM_k^f(\tau)R^{-1},\quad k=1,2.$$
Notice that $M^g_1(\tau), M^g_2(\tau)\in SU(2)$, i.e., $\overline{M_k^g(\tau)}^TM_k^g(\tau)=I_{2}, \, k=1,2$,
then
 $$\overline{R}^TR\left(\begin{array}{cc}e^{2\pi i s}&0\\ 0&e^{-2\pi i s}\end{array}\right)=\left(\begin{array}{cc}e^{2\pi i s}&0\\ 0&e^{-2\pi i s}\end{array}\right)\overline{R}^TR.$$
Combining with $s\not\in \frac{1}{2}\mathbb{Z}$, we have that $$\left(\begin{array}{cc}|a|^2+|c|^2&\overline{a}b+\overline{c}d\\ a\overline{b}+c\overline{d}&|b|^2+|d|^2\end{array}\right)=\overline{R}^TR=\left(\begin{array}{cc}|a|^2+|c|^2&0\\0&|b|^2+|d|^2\end{array}\right),$$ namly, $a\overline{b}+c\overline{d}=0$. Hence, $$|\det R|^2=|ad-bc|^2=(|a|^2+|c|^2)(|b|^2+|d|^2),$$ and then
\begin{align*}
w(z)=\log \frac{8|g'(z)|^2}{(1+|g(z)|^2)^2}
=\log \frac{8|ad-bc|^2|f'(z)|^2}{(|{b}|^2+|{d}|^2+(|a|^2+|c|^2)|f|^2)^2}.
\end{align*}
Let $e^{2t}:=\frac{|a|^2+|c|^2}{|{b}|^2+|{d}|^2}$, we have 
$w(z)=v_{t}(z)$. Therefore,  the set of solutions of (\ref{mfe}) such that the monodromy data $(s,r)\in  \mathbb{R}^2\setminus (\frac{1}{2}\mathbb{Z})^2$  is isomorphic to $\mathbb{R}$. 
\end{proof}

Theorem \ref{thm-1.1} (1) is obtained from Theorem \ref{lem-gen-u}.
Notice that $\mathcal{L}(T)$ with $T\in AP$ is completely reducible and the monodromy data $(s,r)\in  (\frac{1}{2}\mathbb{Z})^2$ if and only if $T=0$ (see \cite[Theorem 1.1]{FL2025}\label{thm-intro-1}).
In addition, by \cite[Lemma 4.5]{FL2025}\label{lem-YV},
 any solution of $\mathcal{L}(0)$, up to a constant multiple,  can be expressed as 
$$y(z;a)=e^{-\frac{1}{2}\eta_3z}\frac{\sigma(z-a)\sigma(z+a)}{(\prod_{j=0}^3(z-\frac{\omega_j}{2})^{\frac{1}{2}}}, \quad \text{with}\,\, a\in \mathbb{C},$$
where $\sigma$ is the Weierstrass sigma function.
Therefore, we obtain that 
\begin{theorem}\label{cor-special}
For any $\tau\in\mathbb{H}$, let $u$ be a solution of (\ref{mfe}) and  $\mathcal{L}(u)=\mathcal{L}(T)$ with $T\in AP$, then $u$ is special if and only if $T=0$. 
In particular,  
  all special solutions of (\ref{mfe}) form an one three-parameter family.
Moreover, any special solution of (\ref{mfe}) is even.  
\end{theorem}

\begin{proof}
We only need to show the last statement, i.e., any special solution of (\ref{mfe}) is even.   Let $u$ be a special solution of (\ref{mfe}), then $\mathcal{L}(u)=\mathcal{L}(0)$.  Let $h$ be a developing map of $u$ and $y_0, y_1$ are two solutions of $\mathcal{L}(0)$ such that $h=y_0/y_1$. 
Notice that, up to some constant multiples, 
$$y_i(z)=e^{-\frac{1}{2}\eta_3z}\frac{\sigma(z-z_i)\sigma(z+z_i)}{(\prod_{j=0}^3\sigma(z-\frac{\omega_j}{2})^{\frac{1}{2}}}, \qquad i=0,1,$$
then the developing map 
$$h(z)=\frac{y_0(z)}{y_1(z)}=\frac{\sigma(z-z_0)\sigma(z+z_0)}{\sigma(z-z_1)\sigma(z+z_1)}.$$
Since  $\sigma(-z)=-\sigma(z)$, then $h(-z)=h(z)$ by the transformation law, thus $u(-z)=u(z)$. 
\end{proof}
By Theorem \ref{cor-special}, we get Theorem \ref{thm-1.1} (2).

\section{Multiple Green functions}\label{sec-G2}
In this section, we will study critical points of the multiple Green function $G_2(z_1, z_2)$, which is defined in (\ref{def-G2}). 

 A pair $\{z_1, z_2\}$ is called a \emph{critical point} of $G_2$ if the following equations hold:
\begin{equation}\label{critical-0}
\frac{\partial G_2}{\partial z_1}(z_1, z_2)=0 \qquad \text{and}\qquad \frac{\partial G_2}{\partial z_2}(z_1, z_2)=0.
\end{equation}
It was proved in \cite{LW2010} that 
\begin{equation}
-4\pi \frac{\partial G}{\partial z}(z;\tau)=\zeta(z)-\eta(z),\qquad z\in E_\tau, 
\end{equation}
where $\eta(z)$ is defined by 
\begin{equation}
\eta(r+s\tau)=r\eta_1(\tau)+s\eta_2(\tau), \qquad s, r\in \mathbb{R}. 
\end{equation}
Here, $\eta_j(\tau):=\zeta(z+\omega_j;\tau)-\zeta(z;\tau),\,j=1,2,3$ are the quasi-periods of $\zeta(z;\tau)$.  In particular, we have $\eta(\omega_j)=\eta_j=2\zeta(\frac{\omega_j}{2})$.  By a direct computation, we have 
\begin{align*}
\eta(t z)&=t\eta(z),\qquad \qquad\quad \, \,\,  z\in\mathbb{C},\,\, t\in\mathbb{R}\\
\eta(z_1+z_2)&=\eta(z_1)+\eta(z_2),\qquad z_1, z_2\in \mathbb{C},
\end{align*}
i.e., $\eta(z)$ is $\mathbb{R}$-linear with respect to  
 $z\in \mathbb{C}$,  
then
\begin{align*}
&-4\pi \frac{\partial G_2}{\partial z_1}(z_1, z_2)\\
=&\zeta(z_1-z_2)-\eta(z_1-z_2)-\frac{1}{2}\sum_{j=0}^3\left(\zeta(z_1-\frac{\omega_j}{2})-\eta(z_1-\frac{\omega_j}{2})\right)\\
=&\zeta(z_1-z_2)-\frac{1}{2}\sum_{j=0}^3\zeta(z_1-\frac{\omega_j}{2})+\eta\left(\frac{1}{2}\sum_{j=0}^3(z_1-\frac{\omega_j}{2})-(z_1-z_2)\right)\\
=&\zeta(z_1-z_2)-\frac{1}{2}\sum_{j=0}^3\zeta(z_1-\frac{\omega_j}{2})+\eta(z_1+z_2)-\frac{1}{2}\eta_3, 
\end{align*}
Similarly, 
\begin{align*}
-4\pi \frac{\partial G_2}{\partial z_2}(z_1, z_2)=\zeta(z_2-z_1)-\frac{1}{2}\sum_{j=0}^3\zeta(z_2-\frac{\omega_j}{2})+\eta(z_1+z_2)-\frac{1}{2}\eta_3.
\end{align*}
Hence,  a pair $\{z_1, z_2\}$ is a {critical point} of $G_2$ if and only if the following equations hold:
\begin{equation}\label{critical-2}
\left\{
\begin{aligned}
\zeta(z_1-z_2)-\frac{1}{2}\sum_{j=0}^3\zeta(z_1-\frac{\omega_j}{2})+\eta(z_1+z_2)=\frac{1}{2}\eta_3,\\
\zeta(z_2-z_1)-\frac{1}{2}\sum_{j=0}^3\zeta(z_2-\frac{\omega_j}{2})+\eta(z_1+z_2)=\frac{1}{2}\eta_3,
\end{aligned}\right.
\end{equation}
which implies the pair  $\{z_1, z_2\}$ satisfies 
\begin{equation}\label{critical-1}
2\zeta(z_1-z_2)+\frac{1}{2}\sum_{j=0}^3\left(\zeta(z_2-\frac{\omega_j}{2})-\zeta(z_1-\frac{\omega_j}{2})\right)=0
\end{equation}
Notice that 
    \begin{equation}\label{ccccc16}
        \zeta(2z)-\frac{1}{2}\sum_{k=0}^3\zeta(z-\frac{\omega_k}{2})=\frac{1}{2}\eta_3.
\end{equation}
Indeed, since $\zeta^\prime(z)=-\wp(z)$, the left hand side of the above identity is a constant independent of $z$, which follows from the following identity:
    \begin{equation*}
        4\wp(2z)=\sum_{k=0}^3\wp(z-\frac{\omega_k}{2}).
    \end{equation*}
Here we use $\sum_{k=1}^3e_k=0$. The constant is $\eta_3/2$ can be obtained by analysing the behavior of the left hand side near $z=0$.

By (\ref{ccccc16}), we obtain that (\ref{critical-1}) is equivalent to 
\begin{equation}\label{ccccc17}
    2\zeta(z_2-z_1)+\zeta(2z_1)-\zeta(2z_2)=0.
\end{equation}
We have the following nontrivial result about (\ref{ccccc17}).
\begin{lemma}\cite[Lemma 4.5]{FL2025}
Let $z_1, z_2\in E_\tau\setminus E_\tau[2]$. Then $z_1, z_2$ satisfy (\ref{ccccc17}) if and only if   
$z_2=-z_1$ on $E_\tau$ or  $z_2=z_1-\frac{\omega_j}{2}$ for some $j\in \{1,2,3\}$ on $E_\tau$. That is, $Y=\cup_{j=0}^3Y_j$, where  
 \begin{align*}
Y_0:=&\left\{\{z_1, z_2\}\in \mathrm{Sym}^2(E_\tau \setminus E_\tau[2])\mid  z_1+z_2=0\right\}\\
Y_j:=&\left\{\{z_1, z_2\}\in \mathrm{Sym}^2(E_\tau \setminus E_\tau[2])\mid  z_1-z_2=\frac{\omega_j}{2}\right\}, \quad j=1,2,3,\\
Y:=&\left\{\{z_1, z_2\}\in \mathrm{Sym}^2(E_\tau \setminus E_\tau[2])\mid  z_1\neq z_2 \,\,\text{and}\,\,(\ref{ccccc17}) \,\, \text{holds}.\right\}.
\end{align*}
\end{lemma}

Clearly, $Y_j\cap Y_k=\emptyset$ for $j\neq k$ and $j,k\in \{1,2,3\}$. 
Denoted by $\ell\in \{1,2,3\}$ and $\{\ell, j, k\}=\{1,2,3\}$, we have 
\begin{equation}\label{Y0Yell}
Y_0\cap Y_\ell=\left\{\{\frac{\omega_\ell}{4}, -\frac{\omega_\ell}{4}\},\,\{\frac{\omega_\ell}{4}+\frac{\omega_j}{2}, -\frac{\omega_\ell}{4}+\frac{\omega_j}{2}\} \right\}.
\end{equation}
%
%
%
%
%
%
%
%
Now we can describe the critical points of $G_2$ in terms of $Y$.
\begin{theorem}\label{lem-G2-2}
A pair $\{z_1, z_2\}\in \mathrm{Sym}^2(E_\tau \setminus E_\tau[2])\setminus \Delta_2$  is a critical point of $G_2$ if and only if $\{z_1, z_2\}\in Y$ and one of the following two cases hold:
\begin{enumerate}
\item $\{z_1, z_2\}\in Y_0$; 
\item $\{z_1,z_2\}\in \cup_{j=1}^3Y_j$ and $2z_1$ is a critical point of $G$.  
\end{enumerate}
\end{theorem}

\begin{proof}
Notice that  (\ref{critical-2}) implies  (\ref{ccccc17}), then all critical points of $G_2$ lie in $Y$, thus
a pair $\{z_1, z_2\}\in \mathrm{Sym}^2(E_\tau \setminus E_\tau[2])\setminus \Delta_2$ is a critical point of $G_2$ if and only if the pair $\{z_1, z_2\}\in Y=\cup_{j=0}^3Y_j$ and satisfies  (\ref{critical-2}).

Let $\{z_1, z_2\}\in Y_0$, i.e., $z_2=-z_1$ on $E_\tau$. It is easy to see that the first equation of (\ref{critical-2}) is equivalent to the second equation, then  (\ref{critical-2}) is equivalent to 
(\ref{critical-1}), and in turn, to  (\ref{ccccc17}), which is always satisfied by $\{z_1, -z_1\}$. Hence, all pairs in $Y_0$ are critical points of $G_2$. 

Let $\{z_1, z_2\}\in \cup_{j=1}^3Y_j$, i.e., $z_2=z_1-\frac{\omega_\ell}{2}$ for some $\ell\in \{1,2,3\}$ on $E_\tau$. 
Notice that  $\eta$ is $\mathbb{R}$-linear and $\eta(\frac{\omega_j}{2})=\frac{1}{2}\eta_j=\zeta(\frac{\omega_j}{2})$, combining with (\ref{ccccc16}), we obtain that (\ref{critical-2}) is equivalent to
\begin{equation}\label{critical-4}
\frac{1}{2}\sum_{j=0}^3\zeta(z_1-\frac{\omega_j}{2})-\eta(2z_1)+\frac{1}{2}\eta_3=0.
\end{equation}
By  (\ref{ccccc16}), we obtain that (\ref{critical-4}) is equivalent to 
\begin{equation}\label{critical-3}
-4\pi\frac{\partial{G}}{\partial z}(2z_1)=\zeta(2z_1)-\eta(2z_1)=0,
\end{equation}
i.e., $2z_1$ with $z_1\in E_\tau\setminus E_\tau[2]$ is a critical point of $G(z;\tau)$. 
Therefore, the pair $\{z_1, z_2\}\in \cup_{j=1}^3Y_j$ is a critical point of $G_2$ if and only if  $2z_1$  is a critical point of $G$. 
\end{proof}

By Theorem  \ref{lem-G2-2}, a critical point $\{z_1, z_2\}\in Y$ of $G_2$ is called \emph{special}  if $\{z_1, z_2\}\in Y_0$. In particular, the two points in $Y_0\cap Y_\ell$ are special critical points by  (\ref{Y0Yell}).
In order to classify the critical points in $Y_\ell$ with $\ell\in\{1,2,3\}$, we define the following map 
\begin{equation*}
\begin{aligned}
T^{(\ell)}: Y_\ell &\to \mathbb{C}\\
\{z_1, z_2\}&\mapsto T^{(\ell)}(\{z_1, z_2\}):=\zeta(z_1)+\zeta(z_1-\frac{\omega_\ell}{2})-\zeta(2z_1)+\frac{1}{2}\eta_\ell
\end{aligned},
\end{equation*}
which is clearly well-defined. This map originates from the correspondence between $Y_\ell$ and the apparent space in \cite{FL2025}. 
%
Notice that $(T^{(\ell)})^{-1}(0)=Y_0\cap Y_\ell$, if we define $T^{(0)}: Y_0\to \mathbb{C}$ by $T^{(0)}(Y_0)\equiv 0$, then we obtain a well-defined map $\mathcal{T}$ on $Y$ satisfying $\mathcal{T}|_{Y_\ell}=T^{(\ell)}, \ell=0,1,2,3$.
We have the following conclusion about this map. 
\begin{Theorem}\cite[Theorem 4.8]{FL2025}\label{mapT}
Let $\ell\in \{1,2,3\}$ and $\{j,k,\ell\}=\{1,2,3\}$. The map $T^{(\ell)}$
 is a double covering map with four ramification points 
$$\left\{ \left\{\pm\frac{\omega_j}{4}, \frac{\omega_j}{4}-\frac{\omega_\ell}{2} \right\},  \left\{\pm\frac{\omega_k}{4}, \frac{\omega_k}{4}-\frac{\omega_\ell}{2} \right\}\right\}$$ on $Y_\ell$. Moreover, the corresponding four branch points are exactly four zeros of $Q_\ell(T)$. More precisely,
$$\left(T^{(\ell)}\left(\left\{\pm\frac{\omega_i}{4}, \frac{\omega_i}{4}-\frac{\omega_\ell}{2} \right\}\right)\right)^2=e_i-e_\ell\quad i=j, k.$$
\end{Theorem}


With the double covering maps, we call a critical point $\{z_1, z_2\}\in Y$ of $G_2$ is \emph{trivial} if $\{z_1, z_2\}$ is a ramification point of $T^{(\ell)}$ for some $\ell\in \{1,2,3\}$ and all other critical points are called \emph{nontrivial}. 
Equivalently, a critical point $\{z_1, z_2\}\in Y$ of $G_2$ is \emph{trivial} if $T^{(\ell)}(\{z_1, z_2\})$ is a zero of $Q_\ell(T)$ and \emph{nontrivial} otherwise. 
Notice that  a critical point $\{z_1, z_2\}\in Y$ of $G_2$ is  special  if $\{z_1, z_2\}\in Y_0=\mathcal{T}^{-1}(0)$ and $Q_\ell(0)\neq 0$ for all $\ell\in \{1,2,3\}$, then all special critical points of $G_2$ are nontrivial, we call the rest nontrivial critical points of $G_2$ are \emph{nonspecial}, i.e., the nontrivial critical points lying in $Y\setminus Y_0$. 
By a direct computation (see the proof of \cite[Theorem 4.8]{FL2025} for details), the fiber over $T^{(\ell)}(\{z_1, z_2\})$ is 
\begin{equation}\label{preimage}
(T^{(\ell)})^{-1}\left(T^{(\ell)}(\{z_1, z_2\})\right)=\left\{\{z_1, z_2\}, \{\frac{\omega_j}{2}-z_1,\frac{\omega_j}{2}-z_2\}\right\}.
\end{equation}
By noticing that $G(z;\tau)$ is even, we obtain that $\{z_1, z_2\}$ is a  critical point of $G_2$ if and only if $\{\frac{\omega_j}{2}-z_1,\frac{\omega_j}{2}-z_2\}$ is a critical point of $G_2$.
Therefore, given two nonspecial critical points $\{z_1, z_2\}, \{\hat z_1,\hat  z_2\}\in Y_\ell$, they are called equivalent if
$T^{(\ell)}(\{ z_1, z_2\})=T^{(\ell)}(\{\hat z_1,\hat  z_2\})$, and the equivalent class of $\{z_1, z_2\}$ is denoted by $[\{z_1, z_2\}]$. Since $T^{(\ell)}$ is a branched double cover, each equivalent class consists of  two elements.

Notice that $G(z;\tau)$ is even and doubly periodic, it always has three trivial critical points $\frac{\omega_j}{2}, \,j=1,2,3$. 
 It was proved
 in \cite{LW2010} that $G(z;\tau)$ has at most one pair $\{\theta,-\theta\}$ of nontrivial critical points depending on the choice of $\tau$, i.e., $G(z;\tau)$ has either $3$ or $5$ critical points. Recall the definition of $\mathcal{E}$ in (\ref{def-E}), the critical points of $G_2(z_1,z_2;\tau)$ are completely determined in the following theorem. 

\begin{theorem}\label{thm-critical-G2}
Let $\ell\in \{1,2,3\}$. The critical points of $G_2$ in $Y_\ell$ can be classified as follows.  
\begin{enumerate}
\item  $G_2$ has two special critical points which are the two points in $Y_0\cap Y_\ell$.
\item  $G_2$ has four trivial critical points
which are the four ramification points of $T^{(\ell)}$.
\item If $\tau\in \mathcal{E}$, then $G_2$ has exactly two equivalent classes $[\{\pm \frac{\theta}{2}, \pm \frac{\theta}{2}-\frac{\omega_\ell}{2}\}]$ of nonspecial critical points.
\item If $\tau\notin \mathcal{E}$, then $G_2$ has no nonspecial critical points. 
\end{enumerate}
\end{theorem}

\begin{proof}
Note that (1) is true by the definition, we only need to show (2),(3),(4). 

Let $\{z_1, z_2\}\in Y_\ell$. By Theorem \ref{lem-G2-2}, $\{z_1, z_2\}$ is a critical point of $G_2(z)$ if and only if $2z_1$ is a critical point of $G(z)$. 
%

 If $2z_1=\frac{\omega_{\ell}}{2}$, which is equivalent to  $z_2=-z_1$, 
then $\{z_1, z_2\}\in Y_0\cap Y_\ell$, thus we get two special critical points in this case. 

 If $2z_1=\frac{\omega_{t}}{2}$ with $t\in \{1,2,3\}\setminus \{\ell\}$, then 
$$\{z_1, z_2\}= \{\frac{\omega_t}{4}, \frac{\omega_t}{4}-\frac{\omega_\ell}{2}\}\quad \text{or} \quad  \{-\frac{\omega_t}{4}, -\frac{\omega_t}{4}-\frac{\omega_\ell}{2}\},$$
which are exactly the four ramification points. Hence, we get four trivial critical points, which are exactly the four ramification points, in this case.


If $2z_1\in \{\theta, -\theta\}$ be a nontrivial critical point of $G$, then 
\begin{equation}\label{nontrivial-1}
\{z_1, z_2\}\in \left\{ \{\pm\frac{\theta}{2}, \pm\frac{\theta}{2}-\frac{\omega_\ell}{2}\},\, \{\frac{\omega_j}{2}\mp\frac{\theta}{2}, \frac{\omega_k}{2}\mp\frac{\theta}{2}\}\right\}.
\end{equation}
In particular, $ \{\pm\frac{\theta}{2}, \pm\frac{\theta}{2}-\frac{\omega_\ell}{2}\}$ are equivalent to $\{\frac{\omega_j}{2}\mp\frac{\theta}{2}, \frac{\omega_k}{2}\mp\frac{\theta}{2}\}$, respectively.  Therefore,  $G_2$ has exactly two equivalent classes $[\{\pm \frac{\theta}{2}, \pm \frac{\theta}{2}-\frac{\omega_\ell}{2}\}]$ of nontrivial critical points if $\tau\in \mathcal{E}$, otherwise, $G_2$ has no nontrivial critical points.
\end{proof}
\section{Critical points and the curvature equation}\label{sec-degeneracy}

In this section, we want to show that for any nonspecial critical point  of $G_2$, we can generate a one-parameter family of solutions. 
We first recall the following result about the monodromy data $(s,r)$. 

\begin{Theorem}\cite{FL2025}\label{thm-intro-data}
Let $\ell\in \{1,2,3\}$ and $T\in V_\ell$.  Then  $\mathcal{L}_\ell(T)$ 
is completely reducible if and only if $Q_\ell(T)\neq 0$.  Moreover, up to a sign,  the monodromy data $(s,r)\mod \mathbb{Z}^2$ 
is determined by  
\begin{equation}\label{eqn-intro-sr}
\left\{
\begin{aligned}
r+s\tau&=2a-\frac{1}{2}\omega_\ell-\frac{1}{2}\omega_3\\
r\eta_1+s\eta_2&=\zeta(2a)-\frac{1}{2}\eta_\ell-\frac{1}{2}\eta_3
\end{aligned}\right.,
\end{equation}
where $\{a,a-\frac{\omega_\ell}{2}\}\in (T^{(\ell)})^{-1}(T)$.
\end{Theorem}

\begin{lemma}\label{lem-nontrivial}
Let $T\in V_\ell\setminus\{0\}$, then $\mathcal{L}_\ell(T)$ is unitarizable if and only if $(T^{(\ell)})^{-1}(T)$ is an equivalent class of nonspecial critical points of $G_2$. 
\end{lemma}

\begin{proof}
Let $T\in V_\ell\setminus \{0\}$.  Pick $\{z_1, z_2\}\in Y_\ell$ such that $T=T^{(\ell)}(\{z_1, z_2\})$. Since $T\neq 0$, then $\{z_1, z_2\}\notin Y_0$.   

If $\mathcal{L}_\ell(T)$ is unitarizable, then $Q_\ell(T)\neq 0$. By  Theorem \ref{thm-intro-data},  the monodromy data  $(s,r)\in \mathbb{R}^2$ is determined by (\ref{eqn-intro-sr}), then 
\begin{align*}
\eta(r+s\tau)&=\eta(2z_1-\frac{1}{2}\omega_\ell-\frac{1}{2}\omega_3)=\eta(2z_1)-\frac{1}{2}\eta_\ell-\frac{1}{2}\eta_3\\
\eta(r+s\tau)&=r\eta_1+s\eta_2=\zeta(2z_1)-\frac{1}{2}\eta_\ell-\frac{1}{2}\eta_3,
\end{align*}
thus $\zeta(2z_1)=\eta(2z_1)$,  so $2z_1$ is a critical point of $G(z)$.
Therefore, $\{z_1, z_2\}$ is a nonspecial critical point of $G_2$.

On the other hand, if $\{z_1, z_2\}$ is a nonspecial critical point of $G_2$, then $Q_\ell(T)\neq 0$ and $\zeta(2z_1)=\eta(2z_1)$.  By Theorem \ref{thm-intro-data}, $\mathcal{L}_\ell(T)$ 
is completely reducible and the monodromy data $(s,r)\in \mathbb{C}^2$ is determined by (\ref{eqn-intro-sr}). Notice that $\zeta(2z_1)=\eta(2z_1)$, we have 
\begin{equation}\label{eqn-intro-sr-1}
\left\{
\begin{aligned}
r+s\tau&=2z_1-\frac{1}{2}\omega_\ell-\frac{1}{2}\omega_3,\\
r\eta_1+s\eta_2&=\eta(2z_1-\frac{1}{2}\omega_\ell-\frac{1}{2}\omega_3).
\end{aligned}\right.
\end{equation}
Since $\tau\eta_1-\eta_2=2\pi i\neq 0$ and $\eta(z)$ is $\mathbb{R}$-linearly with respect to $z$, then (\ref{eqn-intro-sr-1}) has a unique solution $(s,r)\in \mathbb{R}^2$.  
Therefore, $\mathcal{L}_\ell(T)$ is unitarizable. 
\end{proof}

Let $u$ be a nonspecial solution of  (\ref{mfe}). If $\mathcal{L}(u)=\mathcal{L}_\ell(T)$ for some $T\in V_\ell$ with $\ell\in \{1,2,3\}$, then $T\neq 0$ by Theorem \ref{cor-special} and we call $u$ is a nonspecial solution of type $\ell$. 
By Theorem \ref{lem-unit} and Lemma \ref{lem-nontrivial}, we get an one-parameter family of nonspecial solutions of (\ref{mfe}) corresponds to an equivalent class of nonspecial critical points of $G_2$, thus the following conclusion about nonspecial solutions  of  (\ref{mfe}) is obtained by Theorem \ref{thm-critical-G2}.
\begin{theorem}\label{cor-nontrivial}
\begin{enumerate}
\item If $\tau\in \mathcal{E}$, then (\ref{mfe}) has exactly two one-parameter families of nonspecial solutions of type $\ell$ for each  $\ell\in \{1,2,3\}$.
\item If $\tau\notin \mathcal{E}$, then (\ref{mfe}) has no nonspecial solutions.
\end{enumerate}
\end{theorem}

By Theorem \ref{cor-nontrivial}, we obtain Theorem \ref{thm-1.1} (3) and finish the proof of Theorem \ref{thm-1.1}.
%
In the end of this second, we consider the symmetry of solutions for (\ref{mfe}). 

\begin{theorem}\label{thm-sys}
Let $\tau\in \mathcal{E}$ and $\ell\in \{1,2,3\}$, then there is a unique solution $u_j$ such that it is symmetric with respect to $\frac{\omega_j}{4}$ for $j\neq 0, \ell$ in each  one-parameter family of non-special solutions of type $\ell$. 
\end{theorem}

\begin{proof}
Without loss of generality, we consider the one-parameter family of nonspecial solutions of type $\ell$ such that the corresponding equivalent class of nontrivial critical points of $G_2$ is the following 
$$  \left\{ \{\frac{\theta}{2}, \frac{\theta}{2}-\frac{\omega_\ell}{2}\},\, \{\frac{\omega_j}{2}-\frac{\theta}{2}, \frac{\omega_k}{2}-\frac{\theta}{2}\}\right\}.
$$
Denote by $T:=T^{(\ell)}(\{\frac{\theta}{2}, \frac{\theta}{2}-\frac{\omega_\ell}{2}\})$. 
By Lemma \ref{lem-even}, we have $q_\ell(z;T)$ has a period $\frac{\omega_\ell}{2}$ 
and 
$q_\ell(z;T)$ is symmetric
    with respect to $\frac{\omega_j}{4}$ and $\frac{\omega_k}{4}$, i.e., $q_\ell(\frac{\omega_j}{4}-z)=  q_\ell(\frac{\omega_j}{4}+z)$ and  $q_\ell(\frac{\omega_k}{4}-z)=  q_\ell(\frac{\omega_k}{4}+z)$.
Let $y(z)$ be a common eigenfunction of $\mathcal{L}_\ell(T)$, 
 that is,
$$y(z+\omega_i)=\lambda_iy(z), \quad \lambda_i\in \mathbb{C}, \, i=1,2,$$ where $\lambda_1=e^{-2\pi is}, \lambda_2=e^{2\pi i r}$ with $(s,r)\in \mathbb{R}^2\setminus (\frac{1}{2}\mathbb{Z})^2$. 
For  $\{\ell,j,k\}=\{1,2,3\}$, we obtain that 
$$\widehat y_\ell(z):=y(z+\frac{\omega_\ell}{2}), \quad  \widehat y_j(z):=y(\frac{\omega_j}{2}-z) \quad \text{and}\quad \widehat y_k(z):=y(\frac{\omega_k}{2}-z)    $$ are also solutions of  $\mathcal{L}_\ell(T)$ and satisfy 
$$\widehat y_\ell (z+\omega_i)=\lambda_i \widehat y_\ell(z),\,\, \widehat y_j(z+\omega_i)=\lambda_i^{-1} \widehat y_j(z) \,\, \text{and}\,\,\widehat y_k(z+\omega_i)=\lambda_i^{-1} \widehat y_k(z),$$ where $i=1,2$. Since $\lambda_1\neq \lambda_1^{-1}$ or $\lambda_2\neq \lambda_2^{-1}$, then $y(z+\frac{\omega_\ell}{2})$ is linearly dependent with $y(z)$, $y(\frac{\omega_j}{2}-z)$ is linearly dependent with $y(\frac{\omega_k}{2}-z)$ and $y(z), y(\frac{\omega_j}{2}-z)$ are linearly independent. 
 Let $u$ be the solution of (\ref{mfe}) with a developing map  $h(z):=\frac{y(z)}{y(\frac{\omega_j}{2}-z)}$, then 
\begin{align*}
u(z)=\log\frac{8|h'(z)|^2}{(1+|h(z)|^2)^2}
=&\log 8 +2\log\frac{|(\log h(z))'|}{|h(z)|+\frac{1}{|h(z)|}}.
\end{align*}
Let $\lambda>0$ and 
\begin{align*}
u_\lambda(z)=\log 8 +2\log\frac{|\lambda h'(z)|}{1+|\lambda h(z)|^2}=\log 8 +2\log\frac{|\lambda(\log  h(z))'|}{\lambda^2 |h(z)|+\frac{1}{| h(z)|}}.
\end{align*}By the proof of Lemma \ref{lem-gen-u}, $\{u_\lambda\}_{\lambda>0}$ is a  one-parameter family of nonspecial solutions of type $\ell$. 

Notice that    $u_\lambda(\frac{\omega_j}{2}-z)=u_\lambda(z)$ if and only if 
\begin{equation}\label{lambdau-3}
\frac{|\lambda(\log  h(z))'|}{\lambda^2 |h(z)|+\frac{1}{| h(z)|}}=\frac{|\lambda(\log  h(\frac{\omega_j}{2}-z))'|}{|\lambda^2 h(\frac{\omega_j}{2}-z)|+\frac{1}{| h(\frac{\omega_j}{2}-z)|}}.
\end{equation}
Since 
$$h(\frac{\omega_j}{2}-z)=\frac{y(\frac{\omega_j}{2}-z) }{y(z)}=\frac{1}{h(z)},$$
 then (\ref{lambdau-3}) is equivalent to 
\begin{equation}\label{lambdau-4}
\frac{|\lambda(\log  h(z))'|}{\lambda^2 |h(z)|+\frac{1}{| h(z)|}}=\frac{|\lambda(\log  h(z))'|}{\frac{\lambda^2}{|h(z)|}+|h(z)|}
\end{equation}
which is equivalent to 
$$\left(|h(z)|^2-1\right)\left(\lambda^2-1\right )=0,$$
i.e., $\lambda=1$.
 Therefore,  $u=u_1$ is the unique solution in $\{u_\lambda\}_{\lambda>0}$ such that it is symmetric with respect to $\frac{\omega_j}{4}$.  The proof for $k$ is the same. 
\end{proof}

\section{Degeneracy criterion of critical points}\label{sec-deg}
In this section, we will prove Theorem \ref{thm-intro-deg}. 
It was proved in \cite{LW2010} that 
\begin{equation*}
-4\pi \frac{\partial G}{\partial z}(z;\tau)=\zeta(z)-r\eta_1(\tau)-s\eta_2(\tau),\quad z=r+s\tau\in E_\tau, \,r, s\in \mathbb{R}. 
\end{equation*}
Denote by 
$$-r\eta_1(\tau)-s\eta_2(\tau)=az+b\overline z$$
By the Legendre relation $\tau\eta_1-\eta_2=2\pi i$, we have 
$$ b=-\frac{\pi}{\mathrm{Im} \tau},\quad a=-b-\eta_1.$$
Write $z=x+iy$ with $x,y\in \mathbb{R}$, we have 
\begin{align*}
2\pi G_x(z)=-\mathrm{Re}(\zeta(z)+az+b\overline{z}),\\
2\pi G_y(z)=\mathrm{Im}(\zeta(z)+az+b\overline{z}),
\end{align*}
and then 
\begin{align*}
2\pi G_{xx}(z)&=\mathrm{Re}(\wp(z)+\eta_1),\\
2\pi G_{xy}(z)&=-\mathrm{Im}(\wp(z)+\eta_1)=G_{yx}(z),\\
2\pi G_{yy}(z)&=-\mathrm{Re}(\wp(z)+\eta_1)-2b.
\end{align*}
Hence
\begin{align*}
\det D^2G(z)
=\frac{1}{4\pi^2}\left(-2b \mathrm{Re}(\wp(z)+\eta_1)-|\wp(z)+\eta_1|^2\right).
\end{align*}
Recall that 
$$G_2(z_1, z_2)=G(z_1-z_2)-\frac{1}{2}\sum_{j=0}^3\left(G(z_1-\frac{\omega_j}{2})+G(z_2-\frac{\omega_j}{2})\right),  $$
combining with $4\wp(2z)=\sum_{j=0}^3\wp(z-\frac{\omega_j}{2}),$ we have 
\begin{align*}
G_{2x_1x_1}&
=\frac{1}{2\pi}\mathrm{Re}(\wp(z_1-z_2)-2\wp(2z_1)-\eta_1)\\
G_{2x_1y_1}&
=\frac{1}{2\pi}\mathrm{Im}(-\wp(z_1-z_2)+2\wp(2z_1)+\eta_1)\\
G_{2x_1x_2}&
=-\frac{1}{2\pi}\mathrm{Re}(\wp(z_1-z_2)+\eta_1)\\
G_{2x_1y_2}&
=\frac{1}{2\pi}\mathrm{Im}(\wp(z_1-z_2)+\eta_1)\\
G_{2y_1y_1}
&=\frac{1}{2\pi}\left[\mathrm{Re}(-\wp(z_1-z_2)+2\wp(2z_1)+\eta_1)+2b\right]\\
G_{2y_1x_2}&
=\frac{1}{2\pi}\mathrm{Im}(\wp(z_1-z_2)+\eta_1)\\
G_{2y_1y_2}&
=\frac{1}{2\pi}\left[\mathrm{Re}(\wp(z_1-z_2)+\eta_1)+2b\right]\\
G_{2x_2x_2}&
=\frac{1}{2\pi}\mathrm{Re}(\wp(z_1-z_2)-2\wp(2z_2)-\eta_1)\\
G_{2x_2y_2}&
=\frac{1}{2\pi}\mathrm{Im}(-\wp(z_1-z_2)+2\wp(2z_2)+\eta_1)\\
G_{2y_2y_2}
&=\frac{1}{2\pi}\left[\mathrm{Re}(-\wp(z_1-z_2)+2\wp(2z_2)+\eta_1)+2b\right].
\end{align*}

Let a pair $\{z_1, z_2\}$ with $z_1,z_2\in E_\tau\setminus E_\tau[2]$ be a critical point of $G_2$, by Theorem \ref{lem-G2-2}, 
one of the following two cases hold:
\begin{enumerate}
\item[Case 1.] $z_2=-z_1$;
\item[Case 2.] $\{z_1,z_2\}\in \cup_{j=1}^3Y_j$ and $2z_1$ is a critical point of $G$. 
\end{enumerate}

Denote by $A:=\wp(2z_1)+\eta_1$. 
\\

\noindent{\bf Case 1.} If $z_2=-z_1$, we have 
\begin{align*}
\det D^2G_2(z_1, z_2)=\left|\begin{array}{cccc}-\mathrm{Re}A& \mathrm{Im}A&-\mathrm{Re}A &\mathrm{Im}A \\
\mathrm{Im}A &\mathrm{Re}A+2b&\mathrm{Im} A& \mathrm{Re}A+2b\\-\mathrm{Re}A& \mathrm{Im}A&-\mathrm{Re}A &\mathrm{Im}A \\\mathrm{Im}A &\mathrm{Re}A+2b&\mathrm{Im} A& \mathrm{Re}A+2b\end{array}\right|\equiv 0,
\end{align*}
the trivial critical point $\{z_1, z_2\}$ is degenerate. In fact, due to  the nondegerate critical points are isolated, we can get this degeneracy property directly.
\\

\noindent{\bf Case 2.}  If $\{z_1,z_2\}\in Y_k\setminus Y_0$, i.e.,   $z_1-z_2=\frac{\omega_k}{2}$, and $2z_1$ is a critical point of $G$ and $2z_1\neq \frac{\omega_k}{2}$ by the proof of Theorem \ref{thm-critical-G2}, then $\wp(z_1-z_2)=e_k$ and $\wp(2z_2)=\wp(2z_1)$.   Denote by $B:=2A-e_k-\eta_1$, we have 
\begin{align*}
&(2\pi)^4\det D^2G_2(z_1, z_2)\\
&\left|\begin{array}{llll}-\mathrm{Re}B& \mathrm{Im}B&-\mathrm{Re}(2A-B) &\mathrm{Im}(2A-B) \\
\mathrm{Im}B &\mathrm{Re}B+2b&\mathrm{Im} (2A-B)& \mathrm{Re}(2A-B)+2b\\-\mathrm{Re}(2A-B)& \mathrm{Im}(2A-B)&-\mathrm{Re}B &\mathrm{Im}B \\  \mathrm{Im}(2A-B) &\mathrm{Re}(2A-B)+2b&\mathrm{Im} B& \mathrm{Re}B+2b\end{array}\right|\\
=&-16|B-A|^2\left(-|A|^2-2b\mathrm{Re}A\right)
\end{align*}
Notice that $\det D^2G(2z_1)
=\frac{1}{4\pi^2}\left(-2b \mathrm{Re}A-|A|^2\right)$ and $B-A=\wp(2z_1)-e_k$, we have 
$$\det D^2G_2(z_1,z_2)=-\frac{4|\wp(2z_1)-e_k|^2}{\pi^2}\det D^2G(2z_1).$$
Notice that  $2z_1\neq \frac{\omega_k}{2}$, then $\wp(2z_1)\neq e_k$, thus 
\begin{align*}
\det D^2G_2(z_1,z_2)>0 \qquad \text{iff}\qquad  \det D^2G(2z_1)<0;
\end{align*}
\begin{align*}
\det D^2G_2(z_1,z_2)=0 \qquad \text{iff}\qquad  \det D^2G(2z_1)=0;
\end{align*}
\begin{align*}
\det D^2G_2(z_1,z_2)<0 \qquad \text{iff}\qquad  \det D^2G(2z_1)>0.
\end{align*}

Therefore, Theorem \ref{thm-intro-deg} is obtained by
 the following degeneracy criterion of critical points for the Green function $G(z)$. 

\begin{Theorem}\cite{CFL2025-1,Lin,LW2017-1}Let $\tau\in \mathbb{H}$. The Green function  $G(z;\tau)$ always has three trivial critical points $\frac{\omega_j}{2}, \,j=1,2,3$.  Moreover, 
\begin{enumerate} 
\item 
if $\tau\in \mathcal{E}$, then $G$ has a pair of nontrivial critical points $\pm\theta$.
\begin{enumerate}
\item[(1a)]  all nontrivial critical points $z$ of $G$ are nondegenerate minimal points, in particular, $\det D^2G(z)>0$;
\item[(1b)]  all trivial critical points $z$ of $G$  are nondegenerate saddle points, i.e., $\det D^2G(z)<0$.
\end{enumerate}
\item if $\tau\in  \mathbb{H}\setminus \mathcal{E}$, then $G$ has no nontrivial critical points. Futhermore,
\begin{enumerate}
\item[(1a)]  if $\tau\in \mathbb{H}\setminus \overline{\mathcal{E}}$, then all trivial critical points of $G$  are nondegenerate.
\item[(1b)]  if $\tau\in \partial{\mathcal{E}}$, then exactly one trivial critical point of $G$ is degenerate and the other two trivial critical points of $G$ are nondegenerate saddle points, i.e., $\det D^2G(z)<0$.
\end{enumerate}
\end{enumerate}
\end{Theorem}

\end{document}